\documentclass[11pt,letterpaper]{amsart}

\usepackage[T1]{fontenc}
\usepackage[utf8]{inputenc}
\usepackage{lmodern}
\usepackage{microtype}
\usepackage{amsmath,amssymb,amsthm,mathtools,mathrsfs}
\usepackage{xcolor}
\usepackage[margin=1.08in]{geometry}
\usepackage{enumitem}
\usepackage{booktabs,tabularx}
\usepackage{tikz-cd}
\usepackage{refcount}
\usepackage[
  backend=biber,
  style=alphabetic,
  sorting=nyt,
  giveninits=true,
  maxbibnames=99,
  maxalphanames=5,
  doi=true,
  isbn=false,
  url=false,
  eprint=true
]{biblatex}
\usepackage[
  colorlinks=true,
  linkcolor=magenta,
  citecolor=cyan,
  urlcolor=magenta
]{hyperref}
\hypersetup{
  pdftitle={Automorphic commutator relations in relative geometric Langlands},
  pdfauthor={Shurui Liu},
  pdfsubject={relative Langlands program, automorphic commutator relation},
  pdfkeywords={relative Langlands program, automorphic commutator relation,
    period sheaves, Hecke correspondences, geometric Satake equivalence,
    fusion, formal multidisks}
}

\setlist[enumerate]{leftmargin=2.2em,itemsep=2pt,topsep=4pt}
\setlist[itemize]{leftmargin=2em,itemsep=2pt,topsep=4pt}
\allowdisplaybreaks
\numberwithin{equation}{section}
\theoremstyle{plain}
\newtheorem{theorem}{Theorem}[section]
\newtheorem{proposition}[theorem]{Proposition}
\newtheorem{lemma}[theorem]{Lemma}

\newtheorem{maintheorem}{Main Theorem}

\newtheorem{maintheoremB}{Main Theorem}

\newtheorem{assumption}[theorem]{Assumption}
\theoremstyle{definition}

\theoremstyle{remark}
\newtheorem{remark}[theorem]{Remark}
\DeclareMathOperator{\Aut}{Aut}
\DeclareMathOperator{\Bun}{Bun}
\DeclareMathOperator{\Corr}{Corr}
\DeclareMathOperator{\Hom}{Hom}
\DeclareMathOperator{\IC}{IC}
\DeclareMathOperator{\PL}{PL}

\newcommand{\boxstar}{\makebox[0pt][l]{$\square$}\raisebox{.115ex}{\hspace{0.09em}\Large$\star$}\ }

\title{Automorphic commutator relations \\ in relative geometric Langlands}
\author{Shurui Liu}
\address{Stanford University}
\email{srliu@stanford.edu}
\date{}
\subjclass[2020]{Primary 14D24; Secondary 11F67, 14H60, 22E57}
\keywords{Relative Langlands program, automorphic commutator relation, period sheaves, Hecke correspondences, geometric Satake equivalence, fusion, formal multidisks}

\begin{document}

\begin{abstract}
We prove the automorphic commutator relation conjectured by Liu and Wang, which is the geometric input for the automorphic Clifford relations in their framework relating higher period integrals to higher derivatives of $L$-functions over function fields. The relation was previously established only under additional cohomological vanishing assumptions and for curves of genus different from one; we remove these assumptions. The main new ingredient is a ``universal-local'' stage between the local and the semi-local settings: we introduce the classifying stack of formal multidisks with two ordered marked sections, and we construct universal-local analogues of the Hecke stacks, unit objects, Hecke actions and special cohomological correspondences. Fusion and purity on this stack yield a universal-local commutator identity, which is proved using only the local assumptions on the Plancherel algebra that enter the conjecture. Specialization to a curve and the six-functor formalism then give the conjecture for every smooth projective curve.
\end{abstract}

\maketitle

\tableofcontents

\section{Introduction}
\label{sec:introduction}
\subsection{Background and motivation}
\subsubsection{Special cycles and higher derivatives of \(L\)-functions}
\label{special} 
$L$-functions are of central importance in number theory and Langlands program.
Over function fields, higher derivatives of \(L\)-functions are governed by special cycles on moduli stacks of shtukas. This phenomenon was discovered by Yun and Zhang in their work on the higher Gross--Zagier formula \cite{yun2017shtukas,yun2019shtukas}: intersection numbers of Heegner--Drinfeld cycles on the moduli stack of shtukas with \(r\) legs compute the \(r\)-th central derivative of the relevant \(L\)-function. Several further instances are now known. The higher Siegel--Weil formula of Feng, Yun and Zhang \cite{feng2024higher} identifies degrees of special cycles on moduli stacks of unitary shtukas with higher central derivatives of the non-singular Fourier coefficients of Siegel--Eisenstein series, and Feng, Howard and Mkrtchyan \cite{feng2025higher} extended it to the coefficients of corank one. The arithmetic volume formulas of Feng, Yun and Zhang \cite{feng2026arithmetic} relate tautological classes on moduli stacks of shtukas to higher derivatives of Artin \(L\)-functions.

The same special-cycle classes appear as the coefficients of the higher theta series introduced in \cite{feng2021higher}, which places their intersection theory within the Kudla program; the modularity results of \cite{feng2023modularity,feng2024modularity,feng2026modularity} lay the foundation for a higher arithmetic theta lift.

\subsubsection{Relative Langlands duality}
The relative Langlands duality of Ben-Zvi, Sakellaridis and Venkatesh \cite{benzvi2024relativelanglandsduality} provides a conjectural framework in which automorphic periods are related to \(L\)-functions through a duality between Hamiltonian spaces for a reductive group and for its dual group. Liu and Wang \cite{LW} and Wang \cite{wang2025specialcycleshtukascategorical} extended this framework, in the strongly tempered setting, to higher derivatives and to intersection formulas on moduli stacks of shtukas. Their approach is sheaf-theoretic and rests on trace constructions, and its central observation is that higher-derivative formulas follow from automorphic Clifford relations satisfied by certain Hecke operators.

This framework gives a conceptual and uniform explanation of the results recalled in \S\ref{special}. In particular, \cite{LW,wang2025specialcycleshtukascategorical} establish intersection formulas for higher derivatives of Rankin--Selberg \(L\)-functions. Subsequent developments include the diagonal-cycle formulas for adjoint \(L\)-functions of Wang \cite{wang2026diagonalcyclesshtukasadjoint} and the arithmetic volume formulas of Wang and Wei \cite{wang2026arithmetic} for arbitrary coweights of split semisimple groups, which extend the results of \cite{feng2026arithmetic} and \cite{feng2026eigenweights}.

\subsubsection{Automorphic commutator relations}
The key geometric input for the automorphic Clifford relations is the automorphic commutator relation, formulated in \cite[Conjecture~4.45]{LW}. It was proved in \cite[Theorem~4.47]{LW} under the additional Assumption~4.46 of \textit{loc.\ cit.}, which requires three auxiliary cohomological vanishing statements and that the curve \(C\) have genus \(g(C)\neq 1\). The main result of this paper, Main Theorem~\ref{thm:main}, establishes the conjecture without any of these additional conditions.

As a consequence, the commutator identities used in the applications above now have a uniform proof. Specializing Main Theorem~\ref{thm:main} to \(X=\mathrm{pt}\) with one representation trivial recovers \cite[Theorems~3.2 and~3.8]{wang2026arithmetic}, the key ingredients in the proof of the main theorem \cite[Theorem~2.13]{wang2026arithmetic}, which had to be proved separately because the hypotheses of \cite[Theorem~4.47]{LW} are not satisfied there. It also removes the genus restriction from the global commutator identity of \cite[Proposition~2.6]{wang2026diagonalcyclesshtukasadjoint}.

\subsection{Notations and Settings}\label{notations1}
We use the conventions of \cite[\S2 and \S4]{LW}.
\subsubsection{Group}
Let \(G\) be a connected split reductive group over \(\mathbb F_q\). 
\subsubsection{$G$-variety}
Let $X=\mathrm{Spec} A$ be an affine smooth $G$-variety together with
a closed embedding $X \subset V$, where $V$ is a finite-dimensional representation of $G$ regarded as a $G$-variety. 
\subsubsection{Curve}
Let \(C\)
be a smooth projective geometrically connected curve over \(\mathbb F_q\).
\subsubsection{Coefficient field}
Let \(k\) be the characteristic-zero \(\ell\)-adic coefficient field,
where \(\ell\nmid q\), and let \(\check G\) be the Langlands dual group
over \(k\). 
\subsubsection{Shift}
For an integer $n$, we use the shift convention \(\langle n\rangle=[n](n)\) for cohomological shift with half Tate twist.
\subsubsection{Period sheaf}
Let $\mathcal P_X=\pi_!k_{\Bun_G^X}$ be the period sheaf, where
$\pi:\Bun_G^X\to\Bun_G$, and \(\Bun_G^X\) classifies a \(G\)-bundle on \(C\) together with a section of its associated \(X\)-bundle.
\subsubsection{Representations} We use $I=[r]$ to denote an ordered set of cardinality $|I|=r$. Let $V^i$ be a finite-dimensional representation of \(\check G\) for $i\in[r]$. We set \(V^I=\boxtimes_{i\in I}V^i\).
\subsubsection{Global Hecke stack and Hecke action}
We denote by $\mathrm{Hk}^{\text{glob}}_{G,I}$ the global Hecke stack, whose $S$-points are
\begin{multline*}
\mathrm{Hk}^{\text{glob}}_{G,I}(S)=\bigl\{(c_I,\mathscr{E}_0,\mathscr{E}_1,\phi)\bigm|
c_I\in C^I(S),\ \mathscr{E}_0,\mathscr{E}_1\in\mathrm{Bun}_G(S),\\
\phi:\mathscr{E}_0|_{C_S-\Gamma_{c_i}}\xrightarrow{\cong}\mathscr{E}_1|_{C_S-\Gamma_{c_i}}\bigr\}.
\end{multline*}
We have 
\[
\mathrm{Corr}_{\mathrm{Hk}^{\text{glob}}_{G,I},\mathrm{IC}_{V^I}}(\mathcal{P}_X,\mathcal{P}_X\langle d_I - 2|I|\rangle)\cong \mathrm{Hom}^0(\mathrm{IC}_{V^I}\boxstar \mathcal{P}_X, \mathcal{P}_X\langle d_I - 2|I|\rangle),
\]
where $\mathrm{Corr}$ denotes the set of cohomological correspondences reviewed in \S\ref{convention:corr} and $\boxstar$ is the Hecke action, constructed in \cite{arinkin2022stacklocalsystemsrestricted} and reviewed in \cite[\S 4.3.2]{LW}.
\subsubsection{Automorphic cohomological correspondence}
For local special correspondences \(c_{V^i}^{\text{loc}}\) of degrees \(d_i\) (defined in \S\ref{convention:plancherel}),
the local-to-global construction of \cite[Definition 4.41]{LW} produces a global cohomological correspondence
\[
c_{V^I}\in\mathrm{Corr}_{\mathrm{Hk}^{\text{glob}}_{G,I},\mathrm{IC}_{V^I}}(\mathcal{P}_X,\mathcal{P}_X\langle d_I - 2|I|\rangle),
\]
where \(d_I=\sum_i d_i\).

\subsubsection{Assumptions}
We assume \cite[Assumption 4.43]{LW} in \cite[Setting 4.42]{LW}, which is recalled in \S\ref{assumption4.42}. In particular,
we fix \(x\in[r-1]\), and by assumption, there exists a unique element $$c_{V^x\otimes V^{x+1}}^{\text{loc}}\in \Hom^{0}(V^x\otimes V^{x+1}\langle-d_x-d_{x+1}+2\rangle ,\PL_{X,\hbar})$$ such that \[c_{V^x\otimes V^{x+1}}^{\text{loc}}\cdot \hbar = [c_{V^x}^{\text{loc}},c_{V^{x+1}}^{\text{loc}}].\]

Let $[I]:=[r-1]$. The local special correspondences $c_{V^i}^{\text{loc}}$ for $i\neq x,x+1$ and $c_{V^x\otimes V^{x+1}}^{\text{loc}}$ give the automorphic cohomological correspondence $c_{V^{[I]}}$ via the local-to-global construction in \cite[Definition 4.41]{LW}.

Let \(V^{I,\mathrm{sw}}\) be obtained by interchanging the adjacent
factors \(V^x,V^{x+1}\).
\subsection{Main result}
Our main result is that \cite[Conjecture~4.45]{LW} is true.
\begin{maintheorem}[Automorphic commutator relation]
\label{thm:main}
Under Assumption~\ref{assumption4.43}, we have that
\begin{equation*}
\tag{A}\label{eq:main}
c_{V^I}-c_{V^{I,\mathrm{sw}}}
=\Delta^{\mathrm{Hk}}_{x,x+1,!}\,c_{V^{[I]}}
\end{equation*}
in
$
\Corr_{\mathrm{Hk}^{\text{glob}}_{G,I},\,\IC_{V^I}}
\bigl(\mathcal P_X,\mathcal P_X\langle d_I-2|I|\rangle\bigr),
$
where \(\Delta^{\mathrm{Hk}}_{x,x+1}\) is the Hecke diagonal that merges the two legs. 
\end{maintheorem}

\subsection{Idea of the proof: universal-local geometry}
\subsubsection{Reduction to $|I|=2$ and the semi-local case}
The main difficulty is already present for two legs, so we will assume $|I|=2$ for now. Namely, we want to prove that $\epsilon^{\text{glob}}:=c_{V^I}-c_{V^{I,\mathrm{sw}}}
-\Delta^{\mathrm{Hk}}_{x,x+1,!}\,c_{V^{[I]}}$ is zero. One important observation in \cite{LW} is that although the statement of the automorphic commutator relation is global, it suffices to prove its semi-local analogue $\epsilon^{\text{semi-loc}}=0$ (\cite[\S4.6.3]{LW}), where semi-local refers to Beilinson--Drinfeld moduli spaces, i.e. local moduli spaces for a formal disk centered at a moving point on a fixed curve $C$ (see \cite[\S4.3]{LW}).
\subsubsection{Original approach in \cite{LW}}
The collision gives Cartesian squares
$$
\begin{tikzcd}
(LX/L^+G)_C \arrow[r,"i"] \arrow[d] & (LX/L^+G)_{C^2} \arrow[d] & (LX/L^+G)_{C^2-\Delta_C} \arrow[d] \arrow[l,"j"] \\
C \arrow[r,"i"]                     & C^2                       & C^2-\Delta_C \arrow[l,"j"]                      
\end{tikzcd}
$$
It is straightforward to see that $j^\ast\epsilon^{\text{semi-loc}}=0$, and Assumption~\ref{assumption4.43} on the local Plancherel algebra implies $i^\ast\epsilon^{\text{semi-loc}}=0$. Excision expresses $\epsilon^{\text{semi-loc}}$ as a Gysin pushforward $\mathrm{Gys}(\epsilon_0)$, and by the six-functor formalism, one obtains \(c_1(T_C)\cdot \textrm{adj}(\epsilon_0)=0\), where $\textrm{adj}$ is an isomorphism. Therefore, if $c_1(T_C)\cdot$ is injective, we can conclude the proof. This is why additional assumptions were needed for the original approach in \cite[\S 4.6.4]{LW}.
\subsubsection{New approach in this article}
The local-global formalism via semi-local in \cite{LW} is not strong enough to investigate collisions and fusions.

\textit{We propose a fourth stage, ``universal-local,''} and use the concept of formal multidisks with two ordered marked sections to investigate fusion phenomena locally. We define universal-local versions of the objects, spaces, sheaves, and Hecke actions in \cite[\S 4.3]{LW} and \cite{benzvi2024relativelanglandsduality}.

Their collision geometry provides input for fusion arguments in sheaf theory and allows us to prove a universal-local commutator identity, using only the local assumption on the injectivity of $\hbar\cdot(-)$, without an additional global assumption on the injectivity of $c_1(T_C)\cdot(-)$.

\begin{maintheorem}[Universal-local commutator identity]
\label{thm:universal}
Under Assumption~\ref{assumption4.43} we have that
\begin{equation*}
\tag{B}\label{eq:universal_intro}
c_{V,W}^{\text{univ-loc}}-c_{W,V}^{\text{univ-loc}}=\mathrm{Gys}(c_{V\otimes W}^{\text{loc}}),
\end{equation*}
where $c_{V,W}^{\text{univ-loc}}$ and $c_{W,V}^{\text{univ-loc}}$ are defined by Proposition~\ref{lem:spreading}.
\end{maintheorem}
Then the six-functor formalism provides local-to-global compatibility and deduces Main Theorem~\ref{thm:main} from Main Theorem~\ref{thm:universal}.
\subsubsection{Related ideas in literature}
Our definition of formal multidisks with two ordered marked sections is a special case of \cite[Definition 2.25]{Nafcha}. In \cite[Remark 4.49]{LW}, the authors hoped but failed to prove the conjecture by reducing to the $[C=\mathbb{A}^1/\mathbb{G}_m]$ case. One could view Lemma~\ref{lem:purity} as a realization of this idea.

\subsection{Organization of the paper}
Section~\ref{sec:local} fixes the local notation and recalls relevant settings of \cite{LW}.
Section~\ref{sec:multidisks} defines formal multidisks with two marked sections and introduces classifying stack $\mathcal{M}_2$ of formal disks with marked sections, partial Hecke correspondences and the unit object analogous to those in the setting with one marked section. Section~\ref{sec:geom} studies the geometry of $\mathcal{M}_2$ and collisions. Section~\ref{sec:hecke} studies the universal-local Hecke action and special correspondences. Section~\ref{sec:universal} states and proves the universal-local commutator identity. Section~\ref{sec:global} completes the proof of Main Theorem~\ref{thm:main}, by passing from universal-local to semi-local and from semi-local to global.

\subsection{Generative AI usage}
The proof of \cite[Conjecture 4.45]{LW} was initially generated by the mathematical reasoning agent Rethlas~\cite{Rethlas}, enhanced by OSVerify~\cite{osv}, an agentic verification framework. The author verified the proof, reorganized it into a paper and rewrote certain proofs for clearer exposition. Then the author used Codex to check typos and improve the writing. Details of generative AI usage are disclosed in Appendix~\ref{app:ai} for transparency. The author is responsible for the correctness of the paper.
\subsection*{Acknowledgement}
The author would like to thank Zeyu Wang for the collaboration on \cite{LW}, many meaningful discussions, and his feedback on this manuscript. The author is grateful to Fan Nie, Haotian Ye, Zihao Wang, Stefano Ermon and James Zou for collaboration on OSVerify, and to Haocheng Ju, Jiedong Jiang, Guoxiong Gao, Yuefeng Wang, Zeming Sun, Leheng Chen, Bin Wu, Liang Xiao and Bin Dong for collaboration on Rethlas. The author is deeply grateful to his advisor Xinwen Zhu for his constant support.

 \section{Notations and Conventions}
\label{sec:local}

We retain the conventions and setup of \S\ref{notations1}.

\subsection{Sheaf theory}
We use the sheaf theory on étale stacks of \cite[\S4.1]{LW}, denoted by $\mathrm{Shv}(-)$.

\subsection{Cohomological correspondences}\label{convention:corr}
For a correspondence \(Y\xleftarrow{h_L}H\xrightarrow{h_R}Y\) of stacks, we define a cohomological correspondence $c$ from \(\mathcal F\) to
\(\mathcal G\)  (supported on $H$ with
kernel \(K\)) to be an element  
\[
c\in \mathrm{Hom}^0(h_R^*\mathcal F\otimes K,h_L^!\mathcal G)\cong \mathrm{Hom}^0(h_{L,!}(h_R^*\mathcal F\otimes K),\mathcal G).
\]
See \cite[\S4.2]{LW} for the functorialities of cohomological correspondences, including pullback and pushforward.
\subsection{Geometric Setup}
Now we recall the geometric setup in \cite[\S4.3]{LW}.
\subsubsection{Formal disk}\label{convention:formal-disc}
Let \(D=\operatorname{Spf}\mathbb F_q[\![t]\!]\), and let \(\operatorname{Aut}(D)\) denote the group of automorphisms preserving the origin.
\subsubsection{Loop spaces}
We put $Y_1=LX/(L^+G\rtimes\operatorname{Aut}(D)),
P_1=L^+X/(L^+G\rtimes\operatorname{Aut}(D)),$ where $LX(R)=X(R(\!(t)\!)), L^+X(R)=X(R[\![t]\!])$ are the \textit{loop space} and \textit{positive loop space}. 

There is a natural map $i_1:P_1\hookrightarrow Y_1$, and the \textit{unit object} $\delta_1=i_{1,*}k_{P_1}.$ 
\subsubsection{Local Hecke action}
There is a natural action of $\mathrm{Sat}_{G,\hbar}:=\mathrm{Shv}(\mathrm{Hk}^{\text{loc}}_G)$ on $\mathrm{Sat}_{X,\hbar}:=\mathrm{Shv}(Y_1)$, where $\mathrm{Hk}^{\text{loc}}_G:=L^+G\backslash LG/(L^+G\rtimes\operatorname{Aut}(D))$ is the \textit{loop-rotated local Hecke stack}.

\subsubsection{Plancherel algebra}\label{convention:plancherel}
The \textit{loop-rotated Plancherel algebra} \(\PL_{X,\hbar}\) is then defined to be 
$$\PL_{X,\hbar}=\underline{\mathrm{End}}_{\mathrm{Sat}_{G,\hbar}}(\delta_1) \in \mathrm{Sat}_{G,\hbar},$$
the internal endomorphism algebra of the unit object. In particular, we have that
\begin{equation}
\label{eq:local-hom}
\operatorname{Hom}^0_{\mathrm{Sat}_{G,\hbar}}
  (\operatorname{IC}_V*\delta_1,\delta_1\langle d\rangle)
\simeq
\operatorname{Hom}^0_{\check G}
  (V\langle-d\rangle,\mathrm{PL}_{X,\hbar}).
\end{equation}

Following \cite[Definition 4.40]{LW}, we define a \textit{local special cohomological correspondence} of type $V\in\mathrm{Rep}(\check{G})^{\heartsuit,\omega}$ of degree $d$ to be an element
\[
c^{\text{loc}}_V \in \mathrm{Hom}^0_{\check{G}}(V \langle -d \rangle,\mathrm{PL}_{X,\hbar}).
\]

Multiplication in \(\PL_{X,\hbar}\) corresponds to composition of
local Hecke correspondences. 

The parameter \(\hbar\) is identified with the first Chern class of the tangent character of \(\Aut(D)\), with the
normalization of \cite[\S2.11]{LW}.

\subsection{Assumptions on Plancherel algebras}\label{assumption4.42}
The automorphic commutator relation conjecture \cite[Conjecture 4.45]{LW} requires two assumptions on Plancherel algebras (\cite[Setting 4.42, Assumption 4.43]{LW}).
\begin{assumption}\label{assumption4.43}
We assume the following.
\begin{enumerate}
    \item \label{eq:injectivity} The map
    \begin{multline*}
    \hbar:\operatorname{Hom}^0_{\check G} (V\otimes W\langle-d_V-d_W+2\rangle,\mathrm{PL}_{X,\hbar})\\
    \rightarrow \operatorname{Hom}^0_{\check G} (V\otimes W\langle-d_V-d_W\rangle,\mathrm{PL}_{X,\hbar})
    \end{multline*}
    is injective. 
    \item \label{eq:local-bracket} There exists $c_{V\otimes W}^{\text{loc}}\in \Hom^{0}(V\otimes W\langle-d_V-d_W+2\rangle ,\PL_{X,\hbar})$, such that  $\hbar c_{V\otimes W}^{\text{loc}}=[c_V^{\text{loc}},c_W^{\text{loc}}].$
\end{enumerate}
\end{assumption}
\begin{remark}
When $T^\ast X$ is a hyperspherical $G$-Hamiltonian space (in the sense of \cite[\S 3.5]{benzvi2024relativelanglandsduality}), Assumption~\ref{assumption4.43} is a direct consequence of the local relative Langlands conjecture \cite[Conjecture 8.5.2]{benzvi2024relativelanglandsduality}. The flatness of $\mathrm{PL}_{X,\hbar}$ over $k[\hbar]$, and hence Assumption~\ref{assumption4.43}\eqref{eq:injectivity} in this case, is claimed in \cite[Corollary 1.3]{propertiesplan}.
\end{remark}

\section{Multidisks and Moduli Spaces}
\label{sec:multidisks}
In this section, we define formal multidisks with two ordered marked sections and introduce analogues of the local Hecke stack, Hecke correspondences, and the unit object in this setting.
\subsection{Formal multidisk}
Let \(S=\operatorname{Spec}R\) and \(a,b\in R\). A formal multidisk (with two marked sections) is a formal \(S\)-scheme
\[
D_{a,b,S}:=\operatorname{Spf}\!\left(R[z]^{\wedge}_{(z-a)(z-b)}\right),
\]
equipped with the ordered pair of sections \(s_a\) and \(s_b\) defined by \(z=a\) and \(z=b\).

\subsection{Classifying stack of formal multidisks}
Let \(\mathcal M_2\) denote the stack for the étale topology whose objects over \(S\) are formal \(S\)-schemes with two ordered marked sections, étale locally isomorphic to some formal multidisk $D_{a,b,S}$. Morphisms are continuous isomorphisms over \(S\) preserving each marked section. Pullback is defined by completed base change.

\begin{remark}
The stack $\mathcal M_2$ is a multidisk analogue of the stack of formal disks with one marked section, namely
$$
\mathcal M_1\simeq B\operatorname{Aut}(D),
$$
where $\operatorname{Aut}(D)$ is the group of automorphisms of the formal disk $D$ preserving the origin, defined in \S\ref{convention:formal-disc}.
\end{remark}
We have the natural morphism
\begin{equation}
\label{eq:collision}
d_0:\mathcal M_1\hookrightarrow\mathcal M_2,
\end{equation}
which repeats the marked section. 
By Proposition~\ref{lem:coordinates}, $d_0$ is a closed immersion of finite presentation, and its complement is canonically identified with $\mathcal{M}_1\times\mathcal{M}_1$.

\subsection{Local Hecke stack for formal multidisks}
For \((D;s_1,s_2)\in\mathcal M_2(S)\), write
\(D_i^\times\) for the punctured multidisk at \(s_i\) and
\(D^\times=D_1^\times\cap D_2^\times\).
The local Hecke stack for multidisks \(\mathrm{Hk}_2\to\mathcal M_2\) has fiber
\[
\mathrm{Hk}_2(D;s_1,s_2)
=
\bigl\{(\mathcal E_0,\mathcal E_2,\phi)
\mid
\phi:\mathcal E_0|_{D^\times}
\xrightarrow{\sim}\mathcal E_2|_{D^\times}\bigr\},
\]
where \(\mathcal E_j\) are
\(G\)-bundles on \(D\).
Define \(\mathrm{Hk}_{i|2}\) similarly using \(D_i^\times\) instead of $D^\times$.

\subsection{Partial Hecke correspondences for formal multidisks}
Let \(\mathcal B_2\) be the stack whose $S$-points are 
$$\mathcal B_2(S) = \{(D;s_1,s_2;\mathcal E): (D;s_1,s_2)\in\mathcal M_2(S), \mathcal E\in\operatorname{Bun}_G(D)\}$$
for a test scheme $S$. 
There is a natural projection \(\mathcal B_2\to\mathcal M_2\) forgetting the bundle.

Then we have Hecke correspondence in this setting
\begin{equation}
\begin{tikzcd}
              & \mathrm{Hk}_{i|2} \arrow[ld, "h_L"'] \arrow[rd, "h_R"] &               \\
\mathcal{B}_2 &                                              & \mathcal{B}_2
\end{tikzcd}
\end{equation}
where
\[
\begin{aligned}
h_R(D;s_1,s_2;\mathcal E,\mathcal E',\phi) &=(D;s_1,s_2;\mathcal E),\\
h_L(D;s_1,s_2;\mathcal E,\mathcal E',\phi)&=(D;s_1,s_2;\mathcal E').
\end{aligned}
\]

\subsection{Convolution Hecke stack for formal multidisks}
For \((i,j)=(1,2),(2,1)\), define
\[
\widetilde{\mathrm{Hk}}_{ij}
:=\mathrm{Hk}_{i|2}\times_{h_L,\mathcal B_2,h_R}\mathrm{Hk}_{j|2}.
\]
Over \((D;s_1,s_2)\), its objects are chains
\[
\begin{tikzcd}[column sep=large]
\mathcal E_0 \arrow[r,dashed,"\phi_i"]
  & \mathcal E_1 \arrow[r,dashed,"\phi_j"]
  & \mathcal E_2
\end{tikzcd}
\]
of \(G\)-bundles on \(D\), where
\[
\begin{aligned}
\phi_i&:\mathcal E_0|_{D_i^\times}
  \xrightarrow{\sim}\mathcal E_1|_{D_i^\times},\\
\phi_j&:\mathcal E_1|_{D_j^\times}
  \xrightarrow{\sim}\mathcal E_2|_{D_j^\times}.
\end{aligned}
\]
There are two convolution maps
\[
\begin{tikzcd}[column sep=large]
\widetilde{\mathrm{Hk}}_{12} \arrow[r,"m_{12}"]
  & \mathrm{Hk}_2
  & \widetilde{\mathrm{Hk}}_{21} \arrow[l,"m_{21}"']
\end{tikzcd}
\]
given by
\[
m_{ij}(\mathcal E_0,\mathcal E_1,\mathcal E_2;\phi_i,\phi_j)
=
\bigl(\mathcal E_0,\mathcal E_2;
      \phi_j|_{D^\times}\circ\phi_i|_{D^\times}\bigr).
\]
\begin{remark}
This is an analogue of the ordinary loop-rotated local Hecke stack $\mathrm{Hk}^{\text{loc}}_G$ for formal multidisks.
\end{remark}

\subsection{Unit object}
Over \(\mathcal M_2\), let \(Y_2\) classify a \(G\)-bundle on the formal multidisk $(D,s_a,s_b)$ and a section of its associated \(X\)-bundle on the punctured multidisk $D^\times$. 

Let $$i_2:P_2\hookrightarrow Y_2$$ be the substack which requires that the $X$-section extends to the whole multidisk.
\begin{remark}
The space $Y_2\rightarrow \mathcal{M}_2$ is an analogue of $Y_1:=LX/(L^+G\rtimes \mathrm{Aut}(D))\rightarrow B\mathrm{Aut}(D)$, and $P_2$ is an analogue of $P_1:=L^+X/(L^+G\rtimes \mathrm{Aut}(D))$.
\end{remark}

We define the unit object in this setting by 
\begin{equation}
\label{eq:delta}
\delta_2=i_{2,*}k_{P_2}.
\end{equation}

\subsection{Collision locus}
The restrictions of $(Y_2,P_2,\delta_2)$ to the collision locus are \(Y_1,P_1,\delta_1\) of Section~\ref{sec:local}.
 We write
\begin{equation}
\label{eq:strata}
i:Y_1\hookrightarrow Y_2,\qquad i_P:P_1\hookrightarrow P_2,
\quad j:Y_{2,\mathrm{disj}}\hookrightarrow Y_2.
\end{equation}

Note that \(i,i_P\) are the base changes of \(d_0:\mathcal{M}_1\hookrightarrow \mathcal{M}_2\). 
Therefore, $i,i_P$ are closed immersions of finite presentation and \(j\) is a quasi-compact open immersion by Proposition~\ref{lem:coordinates}.

\section{Geometry of the Moduli Stack of Formal Multidisks \texorpdfstring{$\mathcal{M}_2$}{M2}}
\label{sec:geom}
In this section, we study the geometry of $\mathcal{M}_2$ and deduce its consequences for sheaf theory.
\subsection{Geometry of $\mathcal{M}_2$}
\subsubsection{Coordinate presentation and collision locus of $\mathcal{M}_2$}
\begin{proposition}\label{lem:coordinates}
We have the following description of $\mathcal{M}_2$ and the collision locus.
\begin{enumerate}
\item \label{eq:atlas} The morphism
$\pi:\mathbb A^2_{\mathbb F_q}\longrightarrow\mathcal M_2,$ defined by $
(a,b)\longmapsto (D_{a,b},s_a,s_b),$ is a representable, surjective pro-smooth atlas. Its relation scheme \(\mathbb A^2\times_{\mathcal M_2}\mathbb A^2\) admits an inverse-limit
presentation by affine schemes of finite presentation, smooth over \(\mathbb A^2\) via either projection, with transition morphisms whose fibers are affine spaces.
\item The morphism $ d_0:\mathcal M_1\hookrightarrow\mathcal M_2$ is a representable effective Cartier divisor of finite presentation. Its image is the locus where the two marked sections coincide, and its complement is canonically isomorphic to \(\mathcal M_1\times\mathcal M_1\).
\end{enumerate}
\end{proposition}
\begin{proof}
\textbf{Proof of (1). }
\textbf{We first prove the pro-smoothness of $\pi$.}

\begin{lemma}\label{localcor}
We have the following Cartesian diagram
$$
\begin{tikzcd}
\mathrm{Spec}A \arrow[r,"p"] \arrow[d,"r"] & S=\mathrm{Spec}R \arrow[d,"{(D_{a,b},s_a,s_b)}"] \\
\mathbb{A}^2 \arrow[r,"\pi"]             & \mathcal{M}_2             
\end{tikzcd}
$$
where $s=b-a$ and
\begin{equation}
\label{eq:coordinate-ring}
A = R\bigl[
a_0,b_0,a_1,b_1,a_2,b_2,\ldots,
\bigl(b_0(b_0-sa_1)(b_0+sa_1+s^2b_1)\bigr)^{-1}
\bigr].
\end{equation}
In particular, the maps $p$ and $\pi$ are pro-smooth.
\end{lemma}
\begin{proof}[Proof of Lemma~\ref{localcor}]
We denote the fiber product by $Z:=\mathbb{A}^2\times_{\mathcal M_2}S.$ \par
\textbf{We first describe $Z$ and the map $r$ explicitly.}
For an \(R\)-algebra \(R'\), an \(R'\)-point of \(Z\) consists of a pair
\((\alpha,\beta)\in \mathbb{A}^2(R')\) together with a specified marked isomorphism
\[
\varphi:
(D_{a,b};s_a,s_b)_{R'}
\xrightarrow{\sim}
(D_{\alpha,\beta};s_\alpha,s_\beta).
\]

Translate the source coordinate by setting
\[
t=z-a,\qquad s=b-a,\qquad f=t(t-s).
\]
Let \(u\) denote the standard coordinate on \(D_{\alpha,\beta}\).
The isomorphism \(\varphi\) induces an isomorphism of complete
\(R'\)-algebras
\[
\varphi^\sharp:
R'[u]^\wedge_{(u-\alpha)(u-\beta)}
\xrightarrow{\sim}
R'[t]^\wedge_f.
\]
Define the polynomial
\[
F(t):=\varphi^\sharp(u).
\]
Since \(\varphi^\sharp\) is continuous and \(R'\)-linear, it is
determined by \(F\). Preservation of the ordered marked sections
requires
\[
\alpha=F(0),\qquad \beta=F(s).
\]

Division by the monic polynomial \(f\) gives a unique expansion
\[
F(t)=\sum_{n\geq0}(a_n+b_nt)f^n.
\]
Evaluating at the marked sections yields
\[
\alpha=a_0,\qquad \beta=a_0+b_0s.
\]

By explicit coordinate calculation,  one has that $\varphi$ is a marked isomorphism, i.e. substitution by
\(u\mapsto F(t)\) gives a marked isomorphism, if and only if
$
b_0, b_0-sa_1, b_0+sa_1+s^2b_1
$
are invertible. Therefore, we have that
\[
Z\simeq\operatorname{Spec}
R[a_0,b_0,a_1,b_1,\ldots,(b_0(b_0-sa_1)(b_0+sa_1+s^2b_1))^{-1}],
\]
and under this identification, the map \(r:Z\to \mathbb{A}^2\) is given by
\[
(a_n,b_n)_{n\geq0}
\longmapsto(a_0,a_0+b_0s).
\]

\textbf{Now we prove pro-smoothness of $p$ and $\pi$}.

For \(N\geq1\), set
\[
Z_N=\operatorname{Spec}
R[a_0,b_0,\ldots,a_N,b_N,(b_0(b_0-sa_1)(b_0+sa_1+s^2b_1))^{-1}].
\]
Each \(Z_N\) is an open subscheme of \(\mathbb A_R^{2N+2}\), hence is
smooth and of finite presentation over \(R\). Moreover,
\[
Z_{N+1}\simeq Z_N\times_S\mathbb A_S^2,
\qquad
Z\simeq\varprojlim_{N\geq1}Z_N.
\]
Thus the map $p$, i.e. the base change of \(\pi\) along a standard model $(D_{a,b},s_a,s_b)$, is affine
and pro-smooth.

For an arbitrary object
\((\mathfrak D;s_1,s_2)\) over \(S\), define
\[
\mathcal J=\mathcal I_{s_1}\mathcal I_{s_2},
\qquad
\mathfrak D_N=V(\mathcal J^{N+1}).
\]
The finite-level coordinate functor consists of functions \(v\) on
\(\mathfrak D_N\) for which substitution induces an isomorphism
\[
\mathcal O_S[u]\big/
\bigl((u-s_1^*v)(u-s_2^*v)\bigr)^{N+1}
\xrightarrow{\sim}
(\pi_N)_*\mathcal O_{\mathfrak D_N},
\qquad u\longmapsto v.
\]
Here $\pi_N:\mathfrak D_N\to S$ is the structure morphism. This definition is intrinsic to the marked formal scheme, independent of the choice of coordinates.
On standard charts it is represented by \(Z_N\) above, and marked
isomorphisms identify these functors by pullback of functions.
The resulting identifications satisfy the cocycle condition.
Effective affine descent therefore produces the finite-level
coordinate schemes $\mathcal {Z}_N$ over \(S\). The inverse limit $\varprojlim_N \mathcal{Z}_N$ represents
\(S\times_{\mathcal M_2}\mathbb A^2\) and is pro-smooth over \(S\). Therefore, $\pi$ is representable and pro-smooth.
\end{proof}
\textbf{It remains to prove the surjectivity of $\pi$ and the stated properties of $\mathbb{A}^2\times_{\mathcal M_2}\mathbb{A}^2$.}\par
The map $\pi$ is surjective by definition since every object of $\mathcal{M}_2$ is \'etale locally a standard model $(D_{a,b},s_a,s_b)$ and hence $\pi$ \'etale locally has a section.\par

Taking \(S=\mathbb{A}^2\xrightarrow{\pi}\mathcal{M}_2\)  gives the relation
scheme
\[
\mathcal R:=\mathbb{A}^2\times_{\mathcal M_2}\mathbb{A}^2.
\]
Apply Lemma~\ref{localcor} to $S=\mathbb{A}^2$, and we complete the proof of (1).
\par

\textbf{Now we prove (2).}

Let \(p_1,p_2:\mathcal R\rightrightarrows \mathbb{A}^2\) denote the projections
to the source and target presentations. On \(\mathbb{A}^2\), the collision
locus is defined by \(b-a=0\). On \(\mathcal R\), the preceding
calculation in the proof of Lemma~\ref{localcor} gives
\[
p_2^*(b-a)
=\beta-\alpha
=F(s)-F(0)
=b_0s
=b_0p_1^*(b-a).
\]
Since \(b_0\) is invertible, the pulled-back ideals agree:
\[
p_2^*((b-a))=p_1^*((b-a)).
\]
Their canonical identifications satisfy the cocycle condition,
so the ideal \((b-a)\) descends along the faithfully flat atlas
to an ideal sheaf on \(\mathcal M_2\).

Since \((b-a)\) is an invertible ideal
on \(\mathbb{A}^2\), its descent gives an effective Cartier divisor on \(\mathcal M_2\).
The associated closed immersion is representable and of finite presentation by descent from \(V(b-a)\hookrightarrow \mathbb{A}^2\).

Since completion along $t^2$ is the same as completion along $t$, this collision divisor is canonically identified with \(\mathcal M_1\),
compatibly with arbitrary base change.

On the complement, \(s\) is invertible. The Chinese remainder theorem
then splits the completed algebra into the completions at \(t=0\) and
\(t=s\). Therefore, the complement is identified with \(\mathcal M_1\times\mathcal M_1\).
\end{proof}

\subsubsection{Universal formal disk}
The universal pointed formal disk is the
tautological formal family
\[
\pi_{\mathrm{univ}}:
\mathfrak D_{\mathrm{univ}}\longrightarrow\mathcal M_1,
\]
together with the universal marked section
\[
\sigma:\mathcal M_1\longrightarrow\mathfrak D_{\mathrm{univ}},
\]
with
$\pi_{\mathrm{univ}}\circ\sigma=\operatorname{id}_{\mathcal M_1}.$

For every scheme \(S\) and morphism
\(f:S\to\mathcal M_1\) classifying a pointed formal disk
\((D_S\to S,p:S\to D_S)\), there is a canonical isomorphism of
pointed formal families
\[
\bigl(
S\times_{\mathcal M_1}\mathfrak D_{\mathrm{univ}},
\sigma_S
\bigr)
\simeq
(D_S,p),
\]
where \(\sigma_S\) is the base change of \(\sigma\).
These identifications are compatible with base change and
isomorphisms of pointed formal disks.

The tangent line at the universal marked section is
\[
\mathcal T_{\mathrm{univ}}
:=
\sigma^*T_{\mathfrak D_{\mathrm{univ}}/\mathcal M_1}.
\]
Its pullback along \(f\) is canonically identified with
\[
f^*\mathcal T_{\mathrm{univ}}
\simeq p^*T_{D_S/S}.
\]
\subsubsection{The normal bundle of the collision locus}
\begin{lemma}
\label{lem:normal}
The normal bundle of \(d_0:\mathcal{M}_1\hookrightarrow \mathcal{M}_2\) is canonically isomorphic to $\mathcal{T}_{\mathrm{univ}}$. Therefore, its first Chern
class is \(\hbar\), with the convention of \cite[\S2.11]{LW}.
\end{lemma}
\begin{proof}
Let \(\mathcal I\) be the ideal of the collision divisor, and let
\(\sigma:\mathcal M_1\to\mathfrak D_{\mathrm{univ}}\) be as above. We first construct a
canonical identification
\begin{equation}\label{conormal}
    \mathcal I/\mathcal I^2
\simeq
\sigma^*\widehat\Omega^1_{\mathfrak D_{\mathrm{univ}}/\mathcal M_1}.
\end{equation}

Here \(\widehat\Omega^1\) denotes continuous relative differentials.

On the coordinate atlas as in the proof of Lemma~\ref{localcor}, write \(t=z-a\) and \(s=b-a\).
The collision divisor is defined by \(s=0\), and its conormal line
is generated by \(\bar s\). The cotangent line at the common marked
section is generated by \(dt|_{t=0}\). Define the local identification
by
\[
\bar s\longmapsto dt|_{t=0}.
\]

Under a marked coordinate change \(u=F(t)\), the difference of the
marked coordinates satisfies \(s'=b_0s\). On the collision locus,
$
F(t)=a_0+b_0t+t^2H(t),$ and $F'(0)=b_0.
$
Hence
$
\overline{s'}=F'(0)\bar s,$ and $
du|_{t=0}=F'(0)\,dt|_{t=0}.
$
Therefore, the local identifications are compatible with marked
coordinate changes and descend to the isomorphism \eqref{conormal}. Taking duals, we get
\[
N_{d_0}
\simeq
\bigl(\mathcal I/\mathcal I^2\bigr)^\vee
\simeq
\sigma^*T_{\mathfrak D_{\mathrm{univ}}/\mathcal M_1}.
\]

Under \(\mathcal M_1\simeq B\operatorname{Aut}(D)\) and the convention of \cite[\S2.11]{LW}, this tangent
bundle is classified by
$
B\chi:B\operatorname{Aut}(D)\longrightarrow B\mathbb G_m,
$
and therefore,  we have
$
N_{d_0}\simeq(B\chi)^*\mathcal O(-1),
$ and
$$
c_1(N_{d_0})
=(B\chi)^*c_1(\mathcal O(-1))
=\hbar.
$$
\end{proof}

The following lemma is straightforward by construction and definition.
\begin{lemma}
There is a natural map given by the formal completion at two sections
\[
\gamma_C:C^2\longrightarrow\mathcal M_2
\]
for which the square
\begin{equation}
\label{eq:curve-square}
\begin{tikzcd}[column sep=large,row sep=large]
C \arrow[r,"\operatorname{Coor}_C"] \arrow[d,"\Delta_C"']
  & \mathcal M_1 \arrow[d,"d_0"]\\
C^2 \arrow[r,"\gamma_C"']
  & \mathcal M_2
\end{tikzcd}
\end{equation}
is Cartesian, where \(\operatorname{Coor}_C\) classifies the formal
neighborhood of a point of \(C\), equipped with its marked section.
In particular,
\[
\operatorname{Coor}_C^*N_{d_0}\simeq N_{\Delta_C}\simeq T_C.
\]
\end{lemma}

\subsection{Purity}
\begin{lemma}[Regular support and base change]
\label{lem:support}
The inclusion \(i_2:P_2\hookrightarrow Y_2\) is a representable
closed immersion of finite presentation.
The formation of the unit object \(\delta_2=i_{2,*}k_{P_2}\) is compatible with base change. Namely, 
for any morphism \(g:B\to\mathcal M_2\), form 
the Cartesian squares
\[
\begin{tikzcd}
P_B \arrow[r,"i_B"] \arrow[d,"g_P"] & Y_B \arrow[d,"g_Y"] \arrow[r] & B \arrow[d,"g"]   \\
P_2 \arrow[r,"i_2"]           & Y_2 \arrow[r]           & \mathcal{M}_2
\end{tikzcd}
\]
and then we have a canonical isomorphism
\begin{equation}
\label{eq:support-base-change}
g_Y^*\delta_2\xrightarrow{\sim}\delta_B:=i_{B,\ast}k_{P_B}.
\end{equation}
These isomorphisms are compatible with further base change.
\end{lemma}

\begin{proof}
The assertion about $i_2$ is proved \textit{mutatis mutandis} as in the usual loop space case $i_1:P_1\hookrightarrow Y_1$ (see, e.g., \cite[\S 1.a]{pappas2008twisted}).

The morphism \(i_2\) is proper and hence proper base change gives
\[
g_Y^*\delta_2
=g_Y^*i_{2,*}k_{P_2}
\simeq i_{B,*}g_P^*k_{P_2}
\simeq i_{B,*}k_{P_B}
=\delta_B.
\]
\end{proof}
\begin{remark}
    Take \(B=\mathcal M_1\) and \(g=d_0\), and we get
$
    i^*\delta_2\simeq\delta_1,
$
 where \(i:Y_1\to Y_2\) is induced by
\(d_0\). 
\end{remark}
\begin{remark}
Take \(B=C^2\) and \(g=\gamma_C\), and we get
$
\gamma_{C,Y}^*\delta_2\simeq\delta_{C^2},
$
where \(\gamma_{C,Y}:Y_{C^2}\to Y_2\) is the induced projection.
\end{remark}

\begin{lemma}[Purity]
\label{lem:purity}

There are canonical isomorphisms
\begin{equation}
\label{eq:purity}
i_P^!k_{P_2}\simeq k_{P_1}\langle-2\rangle,
\qquad
i^!\delta_2\simeq\delta_1\langle-2\rangle
\end{equation}
compatible with base change, where $i:Y_1\hookrightarrow Y_2$ and $i_P:P_1\hookrightarrow P_2$ are the inclusions of the collision loci.

Under these identifications, $i^\ast\mathrm{Gys}(-)$ is identified with \(\hbar\cdot (-)\), where for $\mathcal{F}\in\mathrm{Shv}(Y_2)$, the Gysin pushforward is defined to be
\begin{equation}
    \begin{split}
        \mathrm{Gys}: \mathrm{Hom}_{\mathrm{Shv}(Y_1)}(i^\ast\mathcal{F},\delta_1\langle-2\rangle)&\rightarrow \mathrm{Hom}_{\mathrm{Shv}(Y_2)}(\mathcal{F},\delta_2),\\
        (i^\ast\mathcal{F}\xrightarrow{c}\delta_1\langle-2\rangle)&
        \mapsto (\mathcal{F}\xrightarrow{\text{adj}}i_*i^*\mathcal{F}\xrightarrow{i_\ast c}i_\ast\delta_1\langle-2\rangle\xrightarrow{\mathrm{G}_i}\delta_2),
    \end{split}
\end{equation}
and
$\mathrm{G}_i:i_\ast\delta_1\langle-2\rangle\cong i_
\ast i^!\delta_2\rightarrow\delta_2$ by adjunction.
\end{lemma}

\begin{proof}
If \(X=\varnothing\), the assertions are immediate. Assume henceforth
that \(X\ne\varnothing\).

The Cartier divisor of Proposition~\ref{lem:coordinates} has a divisor-classifying map
\begin{equation}
\label{eq:divisor-map}
\rho:\mathcal M_2\longrightarrow[\mathbb A^1/\mathbb G_m],
\end{equation}
and we have Cartesian squares
\begin{equation}
\begin{tikzcd}
P_1 \arrow[d, "i_P"] \arrow[r, "i_1"] & Y_1 \arrow[d, "i"] \arrow[r] & \mathcal M_1 \arrow[d, "d_0"] \arrow[r] & B\mathbb G_m \arrow[d, "i_0"]         \\
P_2 \arrow[r, "i_2"]                  & Y_2 \arrow[r]                & \mathcal M_2 \arrow[r]                  & {{[\mathbb A^1/\mathbb G_m]}      }
\end{tikzcd}
\end{equation}
We want to apply locally pro-smooth base change theorem
\cite[Proposition 4.10]{LW} to these diagrams. Therefore, we need to verify the assumptions of the cited proposition, i.e. we want to find a pro-smooth atlas $\mathcal{J}(X)\twoheadrightarrow P_2$ such that the composition $\mathcal{J}(X)\rightarrow P_2\rightarrow \mathcal{M}_2$ is pro-smooth.

\emph{(1) Constructing a pro-smooth cover.}
Define the $\mathbb{A}^2$-scheme 
$$\mathcal J(X)(R)=\{(a,b,v):a,b\in R, v\in X\bigl(R[t]^\wedge_{(t-a)(t-b)}\bigr)\},$$
for a test scheme $S=\mathrm{Spec}R\rightarrow \mathbb{A}^2$.
Then $\mathcal{J}(X)\rightarrow \mathbb{A}^2$ is pro-smooth surjective. Similarly, $\mathcal J(G)\rightarrow \mathbb{A}^2$ is pro-smooth.

Trivializing the \(G\)-bundle gives
\[
P_2\times_{\mathcal M_2}\mathbb{A}^2
\simeq[\mathcal J(X)/\mathcal J(G)].
\]
Then we have a representable pro-smooth surjection
\[
u:\mathcal J(X)\rightarrow [\mathcal J(X)/\mathcal J(G)]\cong P_2\times_{\mathcal M_2}\mathbb{A}^2\rightarrow P_2.
\]
Note that $\mathcal{J}(X)$ is an affine, hence quasi-compact quasi-separated, scheme. 
We have a commutative diagram
$$
\begin{tikzcd}
\mathbb{A}^2 \arrow[d, "\pi"] \arrow[r, "\widetilde{\rho}"] & \mathbb{A}^1 \arrow[d]        \\
\mathcal{M}_2 \arrow[r, "\rho"]                             & {[\mathbb{A}^1/\mathbb{G}_m]}
\end{tikzcd}
$$
where $\widetilde{\rho}(a,b)= b-a$ is a smooth morphism.
Therefore, we have a commutative diagram
\begin{equation}
    \begin{tikzcd}
\mathcal{J}(X) \arrow[dd, "u" {pos=0.25}, two heads] \arrow[r] & \mathbb{A}^2 \arrow[dd, "\pi" {pos=0.25}] \arrow[r, "\widetilde{\rho}"] & \mathbb{A}^1 \arrow[dd]              &  \\
& P_1 \arrow[ld, hook,"i_P"'] \arrow[r,"p_1"]  & \mathcal{M}_1 \arrow[ld, "d_0"', hook] \arrow[r, "\widetilde{\rho}"] & B\mathbb{G}_m \arrow[ld, "i_0", hook] \\
P_2 \arrow[r, "p_2"]  & \mathcal{M}_2 \arrow[r, "\rho"] & {[\mathbb{A}^1/\mathbb{G}_m]} &
\end{tikzcd}
\end{equation}
where the two bottom squares are Cartesian.
Thus the hypotheses of \cite[Proposition 4.10]{LW} are satisfied.

\smallskip
\emph{(2) Purity and proper base change.}
Let
\[
\alpha_z:k_{B\mathbb G_m}\langle-2\rangle
\xrightarrow{\sim}i_0^!k_{{[\mathbb{A}^1/\mathbb{G}_m]}}
\]
be the zero-section purity isomorphism. Since \(i_0\) is representable
of finite presentation, \cite[Proposition 4.10(1)]{LW} gives
\[
(\widetilde{\rho}\circ p_1)^*i_0^!k_{{[\mathbb{A}^1/\mathbb{G}_m]}}
\xrightarrow{\sim}i_P^!(\rho\circ p_2)^*k_{{[\mathbb{A}^1/\mathbb{G}_m]}}\cong i_P^! k_{P_2}.
\]
Apply \(\alpha_z\), and one proves the first formula in \eqref{eq:purity}:
\[
k_{P_1}\langle-2\rangle \xrightarrow{\sim}i_P^!k_{P_2}.
\]

The morphism \(i_2\) is proper by Lemma~\ref{lem:support}.
Extraordinary proper base change gives
\[
i^!\delta_2
=i^!i_{2,*}k_{P_2}
\simeq i_{1,*}i_P^!k_{P_2}
\simeq\delta_1\langle-2\rangle.
\]
This proves the second formula in \eqref{eq:purity}.

\smallskip
\emph{(3) Self-intersection.}
Let \(\mathcal L\) denote the universal line bundle on
\(B\mathbb G_m\). 
Then the self-intersection formula identifies the map
\[
k_{B\mathbb G_m}\langle-2\rangle
\xrightarrow{\alpha_z}i_0^!k_{{[\mathbb{A}^1/\mathbb{G}_m]}}
\longrightarrow i_0^*k_{{[\mathbb{A}^1/\mathbb{G}_m]}}
\simeq k_{B\mathbb G_m}
\]
with multiplication by \(c_1(\mathcal L)\).
Therefore, by Lemma~\ref{lem:normal}, one has that
\[
(\widetilde{\rho}\circ p_1)^*c_1(\mathcal L)
=p_1^*c_1(N_{d_0})
=p_1^*\hbar.
\]
\end{proof}
\begin{remark}
This is the universal-local version of the calculation with \(c_1(T_C)\) in \cite[Lemma 4.56]{LW}. The locally pro-smooth base change theorem \cite[Proposition 4.10]{LW} implies 
\begin{equation}
    \mathrm{Coor}_C^\ast \hbar = c_1(T_C).
\end{equation}
\end{remark}

\section{Universal-local Hecke Correspondences}
\label{sec:hecke}

\subsection{Partial Hecke stacks and convolution}
We have natural restriction maps for $i=1,2$:
\begin{equation}
    \begin{split}
        \mathrm{res}_i: \mathrm{Hk}_{i|2} &\rightarrow \mathrm{Hk}^{\text{loc}}_G \\
        (D,s_1,s_2,\varphi:\mathcal{E}_0|_{D_i^\times}\xrightarrow{\cong}\mathcal{E}_1|_{D_i^\times})&\mapsto (D,s_i,\varphi:\mathcal{E}_0|_{D_i^\times}\xrightarrow{\cong}\mathcal{E}_1|_{D_i^\times})
    \end{split}
\end{equation}
For $V\in\mathrm{Rep}(\check{G})^{\heartsuit,\omega}$, we pull back the Satake sheaf to obtain $\mathrm{IC}_{V,i}:=\mathrm{res}_i^*\mathrm{IC}_V$.

Recall the convolution Hecke stack from \S\ref{sec:multidisks}, the projection map $p_{i}:\widetilde{\mathrm{Hk}}_{ij}\rightarrow \mathrm{Hk}_{i|2}$ and the convolution map $m_{ij}: \widetilde{\mathrm{Hk}}_{ij}\rightarrow\mathrm{Hk}_2$. For sheaf $\mathcal{F}_i\in\mathrm{Shv}(\mathrm{Hk}_{i|2})$, we define
$$\mathcal{F}_i\star\mathcal{F}_j:= m_{ij,!}(p_i^\ast \mathcal{F}_i\otimes p_j^\ast \mathcal{F}_j).$$

For $V,W\in\mathrm{Rep}(\check{G})^{\heartsuit,\omega}$, define
\begin{equation}
\label{eq:ordered-kernels}
\begin{aligned}
K_{V,W}&= \operatorname{IC}_{V,1}\star\operatorname{IC}_{W,2},\\
K_{W,V}&= \operatorname{IC}_{W,2}\star\operatorname{IC}_{V,1}.
\end{aligned}
\end{equation}
\subsection{Universal-local fusion}
\begin{lemma}[Universal-local fusion comparison]
\label{lem:fusion}
We have the following.
\begin{enumerate}
    \item The canonical identification of \(K_{V,W}\) and \(K_{W,V}\) over the disjoint locus extends uniquely to an isomorphism
$
\eta_{V,W}:K_{V,W}\xrightarrow{\sim}K_{W,V}
$
on \(\mathrm{Hk}_2\).
\item  Its restriction to the collision locus is the local Satake
commutativity constraint. 
\item For every smooth curve \(C\), its pullback
along \(\gamma_C:C^2\to\mathcal M_2\) is the usual fusion comparison.
\end{enumerate}
\end{lemma}

\begin{proof}
\textbf{Proof of (1).}
Let \(\pi: T=\mathbb A^2\to\mathcal M_2\) be the coordinate presentation
of Proposition~\ref{lem:coordinates}, and put
\(T^\circ=T\setminus\Delta\). We construct the comparison over \(T\) and then
descend it along $\pi$.

\emph{Fusion on the coordinate presentation.}
Put \(\mathcal H_T=T\times_{\mathcal M_2}\mathrm{Hk}_2\), and denote
the pulled-back kernels by \(K_{V,W,T}\) and \(K_{W,V,T}\).
Trivializing the input bundle gives the Beilinson--Drinfeld
Grassmannian \(\mathrm{Gr}_T\) and a morphism
\(e:\mathrm{Gr}_T\to\mathcal H_T\).
Set
\[
\mathcal A_{V,W}=e^*K_{V,W,T}[2],
\qquad
\mathcal A_{W,V}=e^*K_{W,V,T}[2].
\]
Let
\(j:\mathrm{Gr}_{T^\circ}\hookrightarrow\mathrm{Gr}_T\) be the
open immersion. Factorization identifies $j^*\mathcal A_{V,W}\cong j^*\mathcal A_{W,V}=:\mathcal{A}^\circ$.
By \cite[Proposition 5.4.2]{Zhu}, we have a unique equivariant isomorphism
\[
\mathcal A_{V,W}\simeq j_{!*}\mathcal A^\circ
\simeq\mathcal A_{W,V},
\]
which descends equivariantly along $e$ and gives an isomorphism

\[
\eta_T:K_{V,W,T}\xrightarrow{\sim}K_{W,V,T}.
\]

\emph{Descent under coordinate changes.}
Put \(R=T\times_{\mathcal M_2}T\), with projections \(q_1,q_2\). By abuse of notation, we also use $q_1,q_2$ to denote the base changed projections $\mathcal{H}_R:=R\times_{\mathcal{M}_2}\mathrm{Hk}_2\rightarrow \mathcal{H}_T$.
The kernels have canonical descent isomorphisms
\[
\theta_{V,W}:q_1^*K_{V,W,T}\xrightarrow{\sim}q_2^*K_{V,W,T},
\qquad
\theta_{W,V}:q_1^*K_{W,V,T}\xrightarrow{\sim}q_2^*K_{W,V,T},
\]
since the restriction and convolution diagrams are intrinsic.
To descend \(\eta_T\), it suffices to prove
\begin{equation}
\label{eq:fusion-descent}
\theta_{W,V}\circ q_1^*\eta_T
=
q_2^*\eta_T\circ\theta_{V,W}.
\end{equation}
Both sides agree over
\(R^\circ=q_1^{-1}(T^\circ)=q_2^{-1}(T^\circ)\).
However, on bounded Satake supports, the kernels and their descent data
descend to finite jet models. 
Then one again uses full faithfulness of intermediate extension to obtain \eqref{eq:fusion-descent} from its restriction to \(R^\circ\).

\textbf{Proof of (2).}
Proper base change identifies the restrictions of \(K_{V,W}\) and
\(K_{W,V}\) to the collision locus with
\(\operatorname{IC}_V\star\operatorname{IC}_W\) and
\(\operatorname{IC}_W\star\operatorname{IC}_V\), respectively.
Under these identifications, the restriction of \(\eta_{V,W}\)
is the commutativity constraint obtained from fusion
\cite[Corollary 5.4.11]{Zhu}.

\textbf{Proof of (3).} Part~(3) is proved similarly.
\end{proof}

\subsection{Universal-local special correspondences}
\subsubsection{Hecke stack with meromorphic $X$-sections}
Let \(\mathrm{Hk}_2^X\) be the Hecke stack with meromorphic
\(X\)-sections. An object consists of a formal multidisk \((D,s_1,s_2)\), two \(G\)-bundles
\(\mathcal E_0,\mathcal E_1\) on \(D\), an isomorphism
\[
\varphi:\mathcal E_0|_{D^\times}
\xrightarrow{\sim}\mathcal E_1|_{D^\times},
\]
and sections
\(\xi_\nu\in\Gamma(D^\times,\mathcal E_\nu\times^G X)\),
for \(\nu=0,1\), satisfying \(\varphi_*\xi_0=\xi_1\).
Define \(\mathrm{Hk}_{i|2}^X\) similarly, requiring the modification \(\varphi\) to be defined on \(D_i^\times\).

We have natural maps \[
\begin{tikzcd}[column sep=large,row sep=large]
Y_2
  & \mathrm{Hk}_2^X
      \arrow[l,"h_R"']
      \arrow[r,"h_L"]
      \arrow[d,"\pi"]
  & Y_2\\
  & \mathrm{Hk}_2
  &
\end{tikzcd}
\]
where
\(h_R,h_L\) remember
\((\mathcal E_0,\xi_0)\) and \((\mathcal E_1,\xi_1)\), respectively,
together with $(D,s_1,s_2)$, and the morphism \(\pi\) forgets the \(X\)-sections.
\subsubsection{Hecke action}
For \(K\in\mathrm{Shv}(\mathrm{Hk}_2)\) with bounded Schubert
support and \(\mathcal F\in\mathrm{Shv}(Y_2)\), define the Hecke action by
\begin{equation}
\label{eq:hecke-kernel-action}
K\boxstar \mathcal F
:=
h_{L,!}\bigl(\pi^*K\otimes h_R^*\mathcal F\bigr)
\in\mathrm{Shv}(Y_2).
\end{equation}
The pushforward is taken on the bounded support of \(\pi^*K\).
We also use \(\boxstar\) for the partial Hecke action defined by the analogous diagram for \(\mathrm{Hk}_{i|2}^X\).
\subsubsection{Universal-local special correspondences}
Using \(\eta_{V,W}\) in Lemma~\ref{lem:fusion}, we identify \(K_{W,V}\) with \(K_{V,W}\) and identify their induced actions on \(\delta_2\).

\begin{proposition}[Universal-local special correspondences]
\label{lem:spreading}
The local special correspondences \(c_V^{\text{loc}},c_W^{\text{loc}}\) canonically
determine 
\begin{equation}
\label{eq:ordered-maps}
c_{V,W}^{\text{univ-loc}},c_{W,V}^{\text{univ-loc}}\in
\operatorname{Hom}^0_{Y_2}
\bigl(K_{V,W}\boxstar\delta_2,
      \delta_2\langle d_V+d_W\rangle\bigr)
\end{equation}
with the following properties.
\begin{enumerate}
\item Their restrictions to the disjoint locus agree.
\item There is a canonical identification
\[
i^*(K_{V,W}\boxstar\delta_2)
  \simeq\operatorname{IC}_{V\otimes W}*\delta_1.
\]
Under this identification, \(i^*c_{V,W}^{\text{univ-loc}}\) and \(i^*c_{W,V}^{\text{univ-loc}}\) are identified with \(c_V^{\text{loc}}c_W^{\text{loc}}\) and
\(c_W^{\text{loc}}c_V^{\text{loc}}\), respectively. 
\item Pullback along \(\gamma_C:C^2\to\mathcal M_2\) gives the two
semi-local special correspondences $c_{V\boxtimes W,\{1,2\}}^{\text{semi-loc}}$ and $c_{W\boxtimes V,\{1,2\}}^{\text{semi-loc}}$ defined in \cite[equation (4.35)]{LW}.
\end{enumerate}
\end{proposition}
\begin{proof}
We first extend a local correspondence to either marked section
and then compose the resulting correspondences.

\emph{Step 1: a Cartesian square of regular loci.}
Fix \(i\in\{1,2\}\). Let
\(h_L,h_R:\mathrm{Hk}_{i|2}^X\to Y_2\) be the output and input maps as before and define $\mathcal R_i :=P_2\times_{Y_2,h_L}\mathrm{Hk}_{i|2}^X \times_{h_R,Y_2}P_2,$ which parametrizes partial modifications for which both \(X\)-sections are regular. 

We have the natural map $p_i:\mathcal R_i\longrightarrow P_2,$
which remembers the output bundle and its regular $X$-section. Define
\(p_{\mathrm{loc}}:\mathcal R_{\mathrm{loc}}\to P_1\)
analogously for the usual local Hecke stack.

Restriction to the formal completion along the \(i\)-th section
defines the horizontal maps in the square
\begin{equation}
\label{eq:patching-square}
\begin{tikzcd}[column sep=large,row sep=large]
\mathcal R_i \arrow[r,"q_i"] \arrow[d,"p_i"']
  & \mathcal R_{\mathrm{loc}} \arrow[d,"p_{\mathrm{loc}}"]\\
P_2 \arrow[r,"r_i"']
  & P_1.
\end{tikzcd}
\end{equation}
We claim that this square is Cartesian.

The claim is local on the base. In a standard coordinate
presentation, write
$
A=R[t]^\wedge_f$, $ f=t(t-s),
$
as in the proof of Lemma~\ref{localcor}
and let \(\tau=t\) or \(t-s\) define the selected section. For every
\(m\geq1\), we have $
A/(\tau^m)
\simeq R[t]/(\tau^m)
\simeq R[\tau]/(\tau^m),
$ and hence $
\widehat A_{(\tau)}\simeq R[[\tau]].$

Beauville--Laszlo gluing
(see \cite[Theorem 1.4.3 and proof of Corollary 1.4.5]{Zhu})
therefore identifies \(G\)-bundles over \(A\) with compatible
\(G\)-bundles over \(A[1/\tau]\) and \(R[[\tau]]\), glued over
\(R((\tau))\). This equivalence also holds for bundles equipped
with regular \(X\)-sections: using a \(G\)-equivariant closed
embedding of \(X\) into a representation (assumed in \S\ref{notations1}), one glues the associated
vector-bundle sections and imposes the defining equations of \(X\). Full faithfulness and functoriality of gluing
show that \eqref{eq:patching-square} is Cartesian, compatibly
with base change.

\emph{Step 2: extension of a local correspondence.}
Fix a representation \(U\), and let \(K_{\mathrm{loc}}\) and
\(K_i\) be the pullbacks of the Satake sheaf $\mathrm{IC}_U$ to
\(\mathcal R_{\mathrm{loc}}\) and \(\mathcal R_i\), respectively.
By construction,
\[
K_i\simeq q_i^*K_{\mathrm{loc}}.
\]
We work throughout on the Satake support of \(U\).
On this support, \(p_{\mathrm{loc}}\) is proper of finite
presentation, and hence \(p_i\) is proper of finite presentation by
\eqref{eq:patching-square}.

By adjunction and proper base change, we have that
for every \(n\in\mathbb Z\),
\begin{equation}
\label{eq:adjunction}
\begin{aligned}
\operatorname{Hom}^0_{Y_1}
 \bigl(\operatorname{IC}_U*\delta_1,\delta_1\langle n\rangle\bigr)
&\simeq
\operatorname{Hom}^0_{P_1}
 \bigl(p_{\mathrm{loc},!}K_{\mathrm{loc}},
       k_{P_1}\langle n\rangle\bigr),\\
\operatorname{Hom}^0_{Y_2}
 \bigl(\mathrm{IC}_{U,i}\boxstar\delta_2,\delta_2\langle n\rangle\bigr)
&\simeq
\operatorname{Hom}^0_{P_2}
 \bigl(p_{i,!}K_i,k_{P_2}\langle n\rangle\bigr).
\end{aligned}
\end{equation}

Proper base change in \eqref{eq:patching-square} gives
\begin{equation}
\label{eq:endpoint-base-change}
p_{i,!}K_i
\simeq r_i^*p_{\mathrm{loc},!}K_{\mathrm{loc}}.
\end{equation}
Given a local correspondence
\[
c:\operatorname{IC}_U*\delta_1
\longrightarrow\delta_1\langle n\rangle,
\]
let
\(\bar c:p_{\mathrm{loc},!}K_{\mathrm{loc}}
\to k_{P_1}\langle n\rangle\)
be its image under the first identification in
\eqref{eq:adjunction}. Pulling back \(\bar c\) defines
\[
p_{i,!}K_i
\xrightarrow{\sim}
r_i^*p_{\mathrm{loc},!}K_{\mathrm{loc}}
\xrightarrow{\,r_i^*\bar c\,}
k_{P_2}\langle n\rangle.
\]
Under the second identification in \eqref{eq:adjunction},
this corresponds to a canonical morphism
$$
c_{U,i}:\mathrm{IC}_{U,i}\boxstar\delta_2
\rightarrow\delta_2\langle n\rangle.
$$

\emph{Step 3: Construction of $c_{V,W}^{\text{univ-loc}}$ and $c_{W,V}^{\text{univ-loc}}$.}
Apply Step~2 to \(c_V^{\text{loc}}\) at leg \(1\) and
\(c_W^{\text{loc}}\) at leg \(2\) to obtain $c_{V,1}: \mathrm{IC}_{V,1}\boxstar\delta_2\rightarrow\delta_2\langle d_V\rangle$ and $c_{W,2}:\mathrm{IC}_{W,2}\boxstar\delta_2\rightarrow\delta_2\langle d_W\rangle$. By proper base change and
the projection formula, we have that 
$$K_{V,W}\boxstar\mathcal{F}\cong \mathrm{IC}_{W,2}\boxstar(\mathrm{IC}_{V,1}\boxstar\mathcal{F})$$
and
$$K_{W,V}\boxstar\mathcal{F}\cong \mathrm{IC}_{V,1}\boxstar(\mathrm{IC}_{W,2}\boxstar\mathcal{F}).$$
Identifying $K_{V,W}\cong K_{W,V}$ by \(\eta_{V,W}\), we obtain the morphisms
$
c_{V,W}^{\text{univ-loc}}, c_{W,V}^{\text{univ-loc}}
$
by composing $c_{V,1}$ and $c_{W,2}$ in the two orders.

Then (1) is straightforward by our construction.

On the collision locus, completion along either marked section
is the entire one-point formal disk. Part~(2) follows by construction and proper base change.

Finally, restriction to the selected formal completion, gluing, and the base-change isomorphisms above are compatible with pullback along \(\gamma_C\). Therefore, under
\eqref{eq:adjunction}, the restriction and extension operations
of \cite[equations (4.29)--(4.30)]{LW} are precisely the pullback
of \(\bar c\) and the inverse adjunction used in Step~2.
Thus the single-leg correspondences specialize to those
of \cite[\S4.6]{LW}. Their ordered compositions are the semi-local
correspondences of \cite[equation (4.35)]{LW}, proving~(3).
\end{proof}

\section{The Universal-local Commutator Identity}
\label{sec:universal}
In this section, we state and prove Main Theorem~\ref{prop:universal}.
\subsection{Gysin map}
Recall that, using the identification \(i^*\bigl(K_{V,W}\boxstar \delta_2\bigr)\simeq\operatorname{IC}_{V\otimes W}*\delta_1\),
purity gives the Gysin map
\begin{equation}
\label{eq:gysin}
\begin{aligned}
\mathrm{Gys}:\ &\operatorname{Hom}^0_{Y_1}
 (\operatorname{IC}_{V\otimes W}*\delta_1,\delta_1\langle d_V+d_W-2\rangle)\\
&\qquad\longrightarrow
\operatorname{Hom}^0_{Y_2}
 (K_{V,W}\boxstar \delta_2,\delta_2\langle d_V+d_W\rangle),\\
u\longmapsto\ &\Bigl(
K_{V,W}\boxstar \delta_2\longrightarrow i_*i^*\bigl(K_{V,W}\boxstar \delta_2\bigr)\\
&\qquad\xrightarrow{i_*u}
i_*i^!\delta_2\langle d_V+d_W\rangle
\longrightarrow\delta_2\langle d_V+d_W\rangle
\Bigr).
\end{aligned}
\end{equation}
By Lemma~\ref{lem:purity}, we have that
\begin{equation}
\label{eq:self-intersection}
i^*\mathrm{Gys}(u)=\hbar u.
\end{equation}

\begin{maintheoremB}[Universal-local commutator identity]
\label{prop:universal}

Under Assumption~\ref{assumption4.43} we have that
\begin{equation}
\label{eq:universal}
c_{V,W}^{\text{univ-loc}}-c_{W,V}^{\text{univ-loc}}=\mathrm{Gys}(c_{V\otimes W}^{\text{loc}}),
\end{equation}
where $c_{V,W}^{\text{univ-loc}}$ and $c_{W,V}^{\text{univ-loc}}$ are defined by Proposition~\ref{lem:spreading}.
\end{maintheoremB}

\begin{proof}
Define \(\epsilon^{\text{univ-loc}}=c_{V,W}^{\text{univ-loc}}-c_{W,V}^{\text{univ-loc}}-\mathrm{Gys}(c_{V\otimes W}^{\text{loc}})\).
By Proposition~\ref{lem:spreading} and the support of the Gysin term, \(j^*\epsilon^{\text{univ-loc}}=0\).
By Assumption~\ref{assumption4.43}\eqref{eq:local-bracket}, Proposition~\ref{lem:spreading}, and \eqref{eq:self-intersection},
\begin{equation}
\label{eq:restriction}
i^*\epsilon^{\text{univ-loc}}=[c_V^{\text{loc}},c_W^{\text{loc}}]-\hbar c_{V\otimes W}^{\text{loc}}=0.
\end{equation}

We have the excision triangle
\[
j_!j^*\bigl(K_{V,W}\boxstar \delta_2\bigr)
\longrightarrow K_{V,W}\boxstar \delta_2
\longrightarrow i_*i^*\bigl(K_{V,W}\boxstar \delta_2\bigr)
\]
by \cite[Proposition 4.12]{LW}. Apply \(\operatorname{Hom}(-,\delta_2\langle d_V+d_W\rangle)\),  by exactness on degree-zero morphisms and \(j^*\epsilon^{\text{univ-loc}}=0\), one obtains
\begin{equation}
\label{eq:excision-lift}
\epsilon_0\in\operatorname{Hom}^0_{Y_1}
 (\operatorname{IC}_{V\otimes W}*\delta_1,\delta_1\langle d_V+d_W-2\rangle)
\end{equation}
such that $$\epsilon^{\text{univ-loc}}=\mathrm{Gys}(\epsilon_0).$$
Equations \eqref{eq:self-intersection} and \eqref{eq:restriction} imply \(\hbar \epsilon_0=0\). By Assumption~\ref{assumption4.43}\eqref{eq:injectivity}, one has $\epsilon_0=0$, and hence $\epsilon^{\text{univ-loc}}=0$.
\end{proof}

\section{Proof of Main Theorem~\ref{thm:main}}
\label{sec:global}
In this section, we prove Main Theorem~\ref{thm:main}.
\subsection{Reduction to the case $|I|=2$}
The reduction from the general case to the two-leg case is carried out in \cite[\S4.6.1]{LW}, without additional assumptions. It follows from the functorialities of cohomological correspondences: compatibility of pushforward with composition and of pullback with composition.

\subsection{From universal-local to semi-local}

Fix a smooth curve \(C\). For \(j=1,2\), define
\[
Y_{C^j}:=C^j\times_{\mathcal M_j}Y_j,
\qquad
P_{C^j}:=C^j\times_{\mathcal M_j}P_j,
\]
using \(\operatorname{Coor}_C:C\to\mathcal M_1\) for \(j=1\)
and \(\gamma_C:C^2\to\mathcal M_2\) for \(j=2\).
These are the semi-local moduli of meromorphic and regular
\(X\)-sections, respectively. Write
\[
g_j:Y_{C^j}\to Y_j,\qquad
g_{P,j}:P_{C^j}\to P_j
\]
for the projections. Let \(i_{j,C}:P_{C^j}\hookrightarrow Y_{C^j}\)
be the base change of \(i_j\), and set
\(\delta_{C^j}:=i_{j,C,*}k_{P_{C^j}}\).
Lemma~\ref{lem:support} gives canonical identifications
\[
g_j^*\delta_j\simeq\delta_{C^j}.
\]

The Cartesian square \eqref{eq:curve-square} induces a Cartesian
square
\begin{equation}
\label{eq:curve-collision-Y}
\begin{tikzcd}[column sep=large,row sep=large]
Y_C \arrow[r,"g_1"] \arrow[d,"\Delta_Y"']
  & Y_1 \arrow[d,"i"]\\
Y_{C^2} \arrow[r,"g_2"']
  & Y_2,
\end{tikzcd}
\end{equation}
where \(\Delta_Y\) is induced by the diagonal
\(\Delta_C:C\hookrightarrow C^2\). In particular, proper base
change gives
\begin{equation}
\label{eq:curve-collision-bc}
g_2^*i_*\simeq\Delta_{Y,*}g_1^*.
\end{equation}

The Hecke correspondences also commute with this base change.
Put
\[
\mathrm{Hk}_{2,C}^X:=C^2\times_{\mathcal M_2}\mathrm{Hk}_2^X.
\]
Both squares in
\begin{equation}
\label{eq:curve-Hecke-bc}
\begin{tikzcd}[column sep=large,row sep=large]
Y_{C^2} \arrow[d,"g_2"']
  & \mathrm{Hk}_{2,C}^X
      \arrow[l,"h_{R,C}"']
      \arrow[r,"h_{L,C}"]
      \arrow[d,"g_{\mathrm{Hk}}"]
  & Y_{C^2} \arrow[d,"g_2"]\\
Y_2
  & \mathrm{Hk}_2^X
      \arrow[l,"h_R"]
      \arrow[r,"h_L"']
  & Y_2
\end{tikzcd}
\end{equation}
are Cartesian; the analogous statement holds for the one-point
and ordered convolution diagrams. All Hecke pushforwards below
are taken on the bounded Satake supports. Let \(K_{V,W,C}\) be
the pullback of \(K_{V,W}\) to
\(\mathrm{Hk}_{2,C}:=C^2\times_{\mathcal M_2}\mathrm{Hk}_2\), and put
\[
\mathcal F:=K_{V,W}\boxstar\delta_2,
\qquad
\mathcal F_C:=K_{V,W,C}\boxstar\delta_{C^2}.
\]
Proper base change in \eqref{eq:curve-Hecke-bc} identifies
\(g_2^*\mathcal F\simeq\mathcal F_C\).
By Lemmas~\ref{lem:fusion} and~\ref{lem:spreading}, the pullbacks
of the two ordered universal-local correspondences are the two
semi-local correspondences.

It remains to check the Gysin term. Put \(d=d_V+d_W\), and denote
the purity isomorphisms by
\[
\alpha_i:\delta_1\langle-2\rangle\xrightarrow{\sim}i^!\delta_2,
\qquad
\alpha_C:\delta_C\langle-2\rangle
  \xrightarrow{\sim}\Delta_Y^!\delta_{C^2}.
\]
Their compatibility is by
Lemma~\ref{lem:purity}. More explicitly, let
\(i_0:B\mathbb G_m\hookrightarrow[\mathbb A^1/\mathbb G_m]\) be the zero section.
The composite \(\rho\circ\gamma_C\) classifies the diagonal
\(\Delta_C\), by \eqref{eq:curve-square}.
Writing \(p_2:P_2\to\mathcal M_2\) for the structural morphism,
we therefore have a diagram with both squares Cartesian:
\[
\begin{tikzcd}[column sep=large,row sep=large]
P_C \arrow[r,"g_{P,1}"] \arrow[d,"\Delta_P"']
  & P_1 \arrow[r] \arrow[d,"i_P"']
  & B\mathbb G_m \arrow[d,"i_0"]\\
P_{C^2} \arrow[r,"g_{P,2}"']
  & P_2 \arrow[r,"\rho\circ p_2"']
  & {[\mathbb A^1/\mathbb G_m]}.
\end{tikzcd}
\]
Here \(\Delta_P\) is induced by \(\Delta_C\).

Their compatibility  gives a commutative diagram
\begin{equation}
\label{eq:curve-purity-compatibility}
\begin{tikzcd}[column sep=large,row sep=large]
g_1^*\delta_1\langle-2\rangle
  \arrow[r,"g_1^*\alpha_i"]
  \arrow[d,"\sim"']
  & g_1^*i^!\delta_2
      \arrow[d,"\eqref{eq:curve-collision-Y}","\sim"']\\
\delta_C\langle-2\rangle
  \arrow[r,"\alpha_C"']
  & \Delta_Y^!\delta_{C^2}.
\end{tikzcd}
\end{equation}

Let
\[
\epsilon_i:
i_*\delta_1\langle-2\rangle
\xrightarrow{i_*\alpha_i}
i_*i^!\delta_2
\xrightarrow{\mathrm{counit}}
\delta_2
\]
be the Gysin morphism, and define
\(\epsilon_C:\Delta_{Y,*}\delta_C\langle-2\rangle
\to\delta_{C^2}\) in the same way using \(\alpha_C\).
By \eqref{eq:curve-collision-bc},
\eqref{eq:curve-purity-compatibility}, and compatibility of the
exchange morphism with the counit, \(g_2^*\epsilon_i\) identifies
with \(\epsilon_C\).

Write \(\mathrm{Gys}_{\text{univ-local}}\) for the Gysin
operation occurring in \eqref{eq:universal}. For a morphism
\[
a:i^*\mathcal F\longrightarrow\delta_1\langle d-2\rangle,
\]
its definition in \eqref{eq:gysin} is the composite
\[
\begin{aligned}
\mathcal F
&\xrightarrow{\mathrm{unit}}i_*i^*\mathcal F\\
&\xrightarrow{i_*a}i_*\delta_1\langle d-2\rangle\\
&\xrightarrow{\epsilon_i\langle d\rangle}
\delta_2\langle d\rangle.
\end{aligned}
\]
Pulling back by \(g_2\), the preceding identifications transform
this composite into
\[
\begin{aligned}
\mathcal F_C
&\xrightarrow{\mathrm{unit}}
\Delta_{Y,*}\Delta_Y^*\mathcal F_C\\
&\xrightarrow{\Delta_{Y,*}(g_1^*a)}
\Delta_{Y,*}\delta_C\langle d-2\rangle\\
&\xrightarrow{\epsilon_C\langle d\rangle}
\delta_{C^2}\langle d\rangle.
\end{aligned}
\]
This is precisely the definition of
\(\mathrm{Gys}_{\text{semi-local}}(g_1^*a)\)
in \cite[equation (4.56)]{LW}. We have therefore proved
\begin{equation}
\label{eq:coev-curve-base-change}
g_2^*\mathrm{Gys}_{\text{univ-local}}(a)
=
\mathrm{Gys}_{\text{semi-local}}(g_1^*a).
\end{equation}

Under the collision identification
\(i^*\mathcal F\simeq\operatorname{IC}_{V\otimes W}*\delta_1\),
take \(a=c_{V\otimes W}^{\mathrm{loc}}\).
Its pullback \(g_1^*a\) is
\(c^{\mathrm{semi\text{-}loc}}_{V\otimes W,\{1\}}\).
Apply \(g_2^*\) to \eqref{eq:universal} and use
Lemma~\ref{lem:spreading} and
\eqref{eq:coev-curve-base-change}, and one has
\begin{equation}
\label{eq:semilocal}
c^{\mathrm{semi\text{-}loc}}_{V\boxtimes W,\{1,2\}}
-
c^{\mathrm{semi\text{-}loc}}_{W\boxtimes V,\{1,2\}}
=
\mathrm{Gys}_{\text{semi-local}}
\bigl(c^{\mathrm{semi\text{-}loc}}_{V\otimes W,\{1\}}\bigr).
\end{equation}
This proves \cite[Conjecture 4.55]{LW} for \(C\).

\subsection{From semi-local to global}
The deduction of Main Theorem~\ref{thm:main} (the global automorphic commutator relation, \cite[Conjecture 4.45]{LW}) from the semi-local commutator relation \eqref{eq:semilocal} (\cite[Conjecture 4.55]{LW}) is given in \cite[\S4.6.2, \S4.6.3]{LW} without additional assumptions, by applying the functorialities of cohomological correspondences and unwinding the definitions (see Definitions 4.40 and 4.41 and equations (4.49), (4.62) and (4.67) in \cite{LW}).

Therefore, this completes the proof of Main Theorem~\ref{thm:main}.

\newpage
\appendix
\section{Details of Generative AI Usage}
\label{app:ai}

Generative AI contributed to the mathematical argument and the preparation
of this article. The proof was initially produced by the agentic system Rethlas \cite{Rethlas} in about 1 hour, enhanced by the verification system
OSVerify \cite{osv}. Then the author verified the draft and rewrote it into a paper. Subsequent generative AI assistance was used to revise the \LaTeX{} presentation and improve the writing.

By contrast, Danus (running on GPT-6-Astra) \cite{liu2026danus} failed to solve it within 5 hours, and the GPT-6-Astra Max web interface failed to solve it in 3 attempts with the same prompt given in \S\ref{app:ai-prompt}.

For transparency, this appendix records the details of AI usage in proof generation and verification.

\subsection{Problem prompt}
\label{app:ai-prompt}

\begin{quote}\itshape
Prove Conjecture 4.45 Automorphic commutator relation in without
Assumption 4.46, of arXiv 2504.00275.
\end{quote}
The logs of the generation agent explicitly interpret this task as retaining
\cite[Setting 4.42]{LW}. 

\subsection{Experiment settings}
\label{app:ai-settings}

The experiment took place on September 22, 2026. The model identifiers
and reasoning-effort settings are listed below. Generation and review used the Codex CLI,
version \texttt{0.155.1}, with live web search enabled. The agent system Rethlas \cite{Rethlas} was adapted for OSVerify \cite{osv}, and the review backend was configured as a decomposing verifier.

\begin{center}
\small
\begin{tabularx}{\linewidth}{@{}lX@{}}
\toprule
Recorded item & Value \\
\midrule
Generator & \texttt{gpt-6-astra}, reasoning effort \texttt{max} \\
Reviewer & \texttt{gpt-6-astra}, reasoning effort \texttt{xhigh} \\
Generation agent & active time $\approx$ 36\,min\\
Verification agent & active time $\approx$ 25\,min\\
Total elapsed time & active time $\approx$ 61\,min\\
\bottomrule
\end{tabularx}
\end{center}

\subsection{Reasoning trajectory}
\label{app:ai-output}

The reasoning trajectory is reconstructed from retained memory entries and log files. We summarize it in Figure~\ref{fig:my_image}.

\begin{figure}[htbp]
    \centering
    \includegraphics[width=0.95\linewidth]{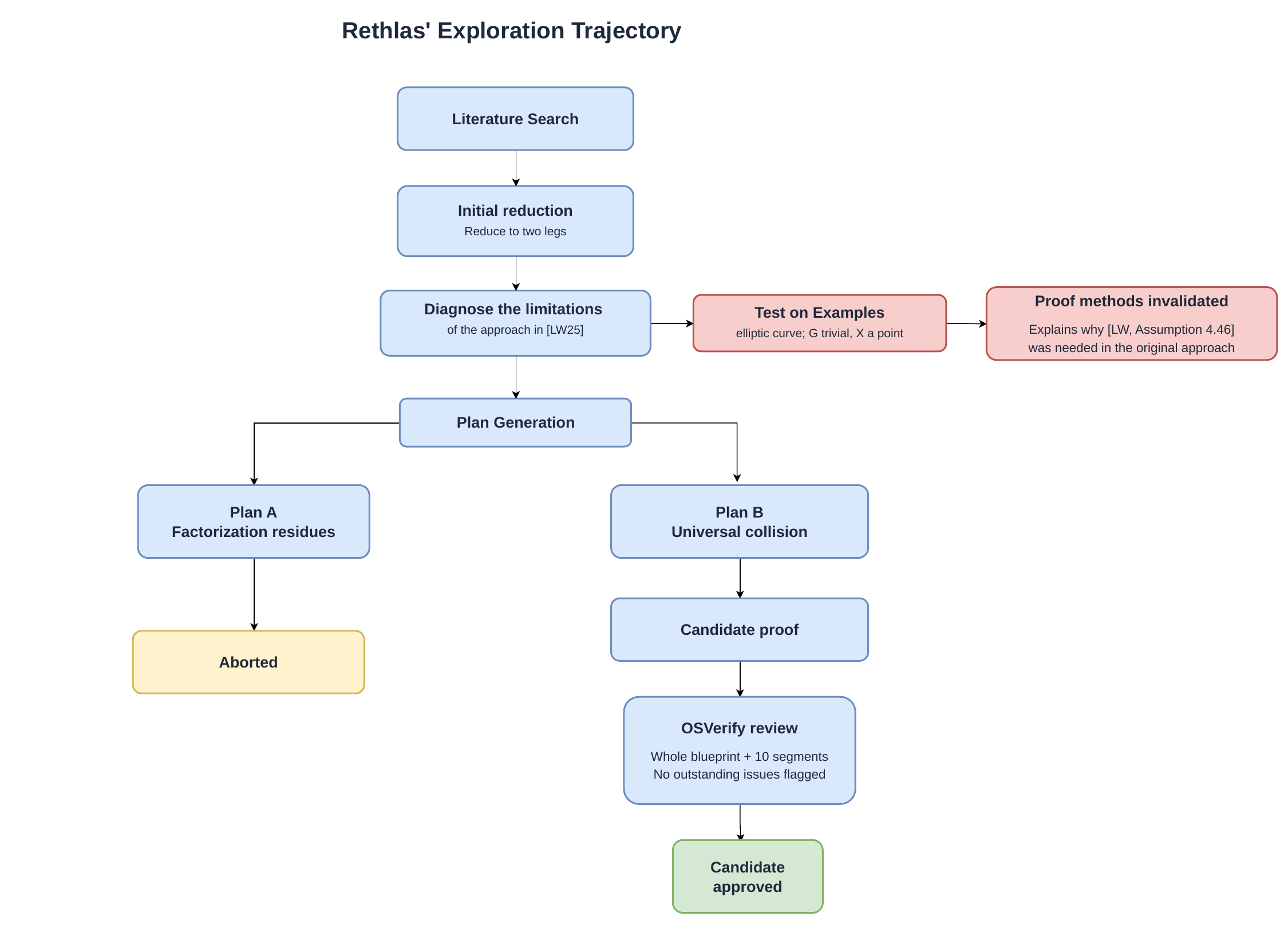}
    \caption{Rethlas--OSVerify reasoning trajectory}
    \label{fig:my_image}
\end{figure}

\begin{enumerate}[label=\textup{\arabic*.}]
\item \textbf{Initial reduction.}
The initial analysis retained \cite[Assumption 4.43]{LW} and reduced the problem to the case of two legs.
\item \textbf{Identifying the limitations of the original approach in \cite{LW}.}
It identified a gap in the naive argument that a correspondence is determined by its restrictions to the diagonal and its complement. As a test case, the record uses the nonzero diagonal
class on the square of an elliptic curve: it restricts to zero on the
complement, while its self-intersection is \(c_1(T_C)=0\). A second
test, with trivial \(G\) and \(X\) a point, showed why local
\(\hbar\)-injectivity need not survive specialization to a curve.
These tests invalidated the proposed proof methods and helped the agent understand why \cite[Assumption 4.46]{LW} was needed in the original approach of \cite{LW}.

\item \textbf{Investigating two approaches.}
Two internal branches investigated ``universal collision'' and ``factorization
residues.'' The first developed a space of
strictly marked formal multidisks. The second investigated whether a
factorization residue could supply a canonical divided commutator.
The record did not establish the comparison needed to apply the
higher-algebra results, and that direct
route was set aside.

\item \textbf{Constructing the universal collision.}
The selected approach constructed an explicit coordinate presentation
for \(\mathcal M_2\), identified its collision divisor with
\(B\Aut(D)\), and computed the normal line as the tangent character.
The resulting self-intersection class is \(\hbar\). The same branch
developed the regular-section condition and the pro-smooth presentations
used for purity.

\item \textbf{Constructing the morphisms and their descent.}
The agent separated descent of the Satake fusion comparison
from construction of the local correspondence maps. The former uses
intermediate extension at finite coordinate stages. The latter uses Beauville--Laszlo
patching and proper base change. They avoid inferring equivalence of derived morphisms solely from equality on classical coordinate charts.

\item \textbf{Assembling the proof.}
The final blueprint applies excision to the universal error term,
writes it as \(\mathrm{Gys}(u)\), and uses self-intersection to obtain
\(\hbar u=0\) and hence \(u=0\). Only after this step does the proof specialize to a curve, pass to the period sheaf,
and compose with the remaining legs. 
\item \textbf{OSVerify finally approved the candidate solution.} OSVerify reviewed the submitted blueprint as a whole and in ten additional segments, and did not flag outstanding issues.
\end{enumerate}
Then the author verified the generated proof, reorganized it into a paper and rewrote certain proofs for clearer exposition.

\clearpage
\printbibliography

@online{LW,
  author       = {Liu, Shurui and Wang, Zeyu},
  title        = {Higher period integrals and derivatives of {$L$}-functions},
  year         = {2026},
  eprinttype   = {arXiv},
  eprint       = {2504.00275v3},
  eprintclass  = {math.NT},
  note         = {Version 3},
}

@online{Nafcha,
  author       = {Nafcha, Elchanan},
  title        = {Nodal degeneration of chiral algebras {I}: global structure and gluing formula},
  year         = {2026},
  eprinttype   = {arXiv},
  eprint       = {2603.30037v1},
  eprintclass  = {math.AG},
  note         = {Version 1},
}

@incollection{Zhu,
  author       = {Zhu, Xinwen},
  title        = {An introduction to affine {Grassmannians} and the geometric {Satake} equivalence},
  booktitle    = {Geometry of moduli spaces and representation theory},
  editor       = {Bezrukavnikov, Roman and Braverman, Alexander and Yun, Zhiwei},
  series       = {IAS/Park City Mathematics Series},
  volume       = {24},
  publisher    = {American Mathematical Society},
  location     = {Providence, RI},
  year         = {2017},
  pages        = {59--154},
  eprinttype   = {arXiv},
  eprint       = {1603.05593v2},
  eprintclass  = {math.AG},
  note         = {Numbered references use arXiv version 2 (2016)},
}

@online{Rethlas,
  author       = {Ju, Haocheng and Gao, Guoxiong and Jiang, Jiedong and Wu, Bin and Sun, Zeming and Liu, Shurui and Chen, Leheng and Wang, Yutong and Wang, Yuefeng and Wang, Zichen and He, Wanyi and Wu, Peihao and Xiao, Liang and Liu, Ruochuan and Dai, Bryan and Dong, Bin},
  title        = {Automated conjecture resolution with formal verification},
  year         = {2026},
  eprinttype   = {arXiv},
  eprint       = {2604.03789},
  eprintclass  = {cs.LG},
}

@online{liu2026danus,
  author       = {Liu, Jihao and Gao, Guoxiong and Sun, Zeming and Wu, Bin and Liu, Shurui and Jiang, Jiedong and Ju, Haocheng and Chen, Leheng and Cheng, Ronnie and Zhang, Xiping and Dong, Bin},
  title        = {{Danus}: orchestrating mathematical reasoning agents with fact-graph memory},
  year         = {2026},
  eprinttype   = {arXiv},
  eprint       = {2607.06447},
}

@misc{osv,
  author       = {Nie, Fan and Ye, Haotian and Liu, Shurui and Wang, Zihao and Ermon, Stefano and Zou, James},
  title        = {One-sided verification under limited expert review},
  year         = {2026},
  note         = {in submission},
}

@misc{propertiesplan,
  title        = {On properties of {Plancherel} algebras},
  howpublished = {Hexagon preprint 2609.00128v1},
  year         = {2026},
  note         = {Contributors: Shurui Liu; \url{https://hexagonmath.org/2609.00128v1}},
}

@online{benzvi2024relativelanglandsduality,
  author       = {Ben-Zvi, David and Sakellaridis, Yiannis and Venkatesh, Akshay},
  title        = {Relative {Langlands} duality},
  year         = {2024},
  eprinttype   = {arXiv},
  eprint       = {2409.04677},
  eprintclass  = {math.RT},
}

@online{arinkin2022stacklocalsystemsrestricted,
  author       = {Arinkin, D. and Gaitsgory, D. and Kazhdan, D. and Raskin, S. and Rozenblyum, N. and Varshavsky, Y.},
  title        = {The stack of local systems with restricted variation and geometric {Langlands} theory with nilpotent singular support},
  year         = {2022},
  eprinttype   = {arXiv},
  eprint       = {2010.01906},
  eprintclass  = {math.AG},
}

@article{pappas2008twisted,
  author       = {Pappas, Georgios and Rapoport, Michael},
  title        = {Twisted loop groups and their affine flag varieties},
  journal      = {Advances in Mathematics},
  volume       = {219},
  number       = {1},
  pages        = {118--198},
  year         = {2008},
  note         = {With an appendix by T. Haines and Rapoport},
}

@online{wang2025specialcycleshtukascategorical,
  author       = {Wang, Zeyu},
  title        = {Special cycle on {Shtukas} and categorical trace},
  year         = {2025},
  eprinttype   = {arXiv},
  eprint       = {2509.05526},
  eprintclass  = {math.NT},
}

@online{wang2026diagonalcyclesshtukasadjoint,
  author       = {Wang, Zeyu},
  title        = {Diagonal cycles on {Shtukas} and the adjoint {$L$}-function},
  year         = {2026},
  eprinttype   = {arXiv},
  eprint       = {2607.25237},
  eprintclass  = {math.NT},
}

@online{wang2026arithmetic,
  author       = {Wang, Zeyu and Wei, Wenqing},
  title        = {Arithmetic volume of {Shtukas} and {Langlands} duality},
  year         = {2026},
  eprinttype   = {arXiv},
  eprint       = {2604.03763},
}

@article{yun2017shtukas,
  author       = {Yun, Zhiwei and Zhang, Wei},
  title        = {Shtukas and the {Taylor} expansion of {$L$}-functions},
  journal      = {Annals of Mathematics (2)},
  volume       = {186},
  number       = {3},
  pages        = {767--911},
  year         = {2017},
}

@article{yun2019shtukas,
  author       = {Yun, Zhiwei and Zhang, Wei},
  title        = {Shtukas and the {Taylor} expansion of {$L$}-functions ({II})},
  journal      = {Annals of Mathematics (2)},
  volume       = {189},
  number       = {2},
  pages        = {393--526},
  year         = {2019},
}

@online{feng2026arithmetic,
  author       = {Feng, Tony and Yun, Zhiwei and Zhang, Wei},
  title        = {Arithmetic volumes of moduli stacks of {Shtukas}},
  year         = {2026},
  eprinttype   = {arXiv},
  eprint       = {2601.18557},
}

@article{feng2024higher,
  author       = {Feng, Tony and Yun, Zhiwei and Zhang, Wei},
  title        = {Higher {Siegel--Weil} formula for unitary groups: the non-singular terms},
  journal      = {Inventiones Mathematicae},
  volume       = {235},
  number       = {2},
  pages        = {569--668},
  year         = {2024},
}

@online{feng2025higher,
  author       = {Feng, Tony and Howard, Benjamin and Mkrtchyan, Mikayel},
  title        = {Higher {Siegel--Weil} formula for unitary groups {II}: corank one terms},
  year         = {2025},
  eprinttype   = {arXiv},
  eprint       = {2507.13473},
}

@online{feng2021higher,
  author       = {Feng, Tony and Yun, Zhiwei and Zhang, Wei},
  title        = {Higher theta series for unitary groups over function fields},
  year         = {2021},
  eprinttype   = {arXiv},
  eprint       = {2110.07001},
}

@online{feng2023modularity,
  author       = {Feng, Tony and Yun, Zhiwei and Zhang, Wei},
  title        = {Modularity of higher theta series {I}: cohomology of the generic fiber},
  year         = {2023},
  eprinttype   = {arXiv},
  eprint       = {2308.10979},
}

@online{feng2024modularity,
  author       = {Feng, Tony and Khan, Adeel A.},
  title        = {Modularity of higher theta series {II}: {Chow} group of the generic fiber},
  year         = {2024},
  eprinttype   = {arXiv},
  eprint       = {2403.19711},
}

@online{feng2026modularity,
  author       = {Feng, Tony and Yun, Zhiwei and Zhang, Wei},
  title        = {Modularity of higher theta series {III}: proof of the modularity conjecture},
  year         = {2026},
  eprinttype   = {arXiv},
  eprint       = {2608.07173},
}

@article{feng2026eigenweights,
  author       = {Feng, Tony},
  title        = {Eigenweights for arithmetic {Hirzebruch} proportionality},
  journal      = {PNAS Nexus},
  volume       = {5},
  number       = {5},
  pages        = {pgag143},
  year         = {2026},
}

\end{document}